\documentclass[12pt]{amsart}
\usepackage{amsmath}
\usepackage{amsthm}
\usepackage{amssymb}
\usepackage{amscd}
\usepackage{hyperref}
\usepackage{color}

\newcommand{\gothic}{\mathfrak}
\newcommand{\p}{{\gothic{p}}}

\newcommand{\m}{{\gothic{m}}}

\newcommand{\8}{{\infty}}

\newcommand{\Ass}{\operatorname{Ass}}

\newcommand{\charac}{\operatorname{char}}
\newcommand{\depth}{\operatorname{depth}}
\newcommand{\height}{\operatorname{ht}}

\newcommand{\pd}{\operatorname{pd}}
\newcommand{\Spec}{\operatorname{Spec}}
\newcommand{\Supp}{\operatorname{Supp}}

\newcommand{\syz}{\operatorname{syz}}
\newcommand{\Tor}{\operatorname{Tor}}
\newcommand{\Ext}{\operatorname{Ext}}

\newcommand{\IPD}{\operatorname{IPD}}

\newcommand{\Hom}{\operatorname{Hom}}
\newcommand{\cok}{\operatorname{coker}}

\newcommand{\inc}{\subseteq}

\newcommand{\length}{\ell}

\newcommand{\e}{\text{e}}
\newcommand{\ehk}{\e_{HK}}

\renewcommand{\hat}{\widehat}

\renewcommand{\phi}{\varphi}

\DeclareMathOperator{\eu}{e_{gHK}^{+}}
\DeclareMathOperator{\ed}{e_{gHK}^{-}}
\DeclareMathOperator{\hk}{e_{gHK}}
\newcommand{\fhk}[1]{f_{gHK}^{#1}}

\DeclareMathOperator{\lc}{H}
\DeclareMathOperator{\CH}{H}

\newcommand{\ses}[3]{0 \to {#1} \to {#2} \to {#3} \to 0}

\DeclareMathOperator{\mo}{\mathbf{mod}}
\newcommand{\ps}[1]{{{}^{n}\!#1}} 
\newcommand{\ul}[1]{\underline{#1}}

\newcommand{\rh}{\operatorname{RHom}}

\newcommand{\dtor}{\stackrel{L}{\otimes}}

\newcommand{\cotimes}{\operatorname{\hat{\otimes}_k}}

\theoremstyle{plain}
\newtheorem{thm}{Theorem}
\newtheorem{cor}[thm]{Corollary}
\newtheorem{prop}[thm]{Proposition}
\newtheorem{lemma}[thm]{Lemma}

\newtheorem{claim}{Claim}

\newtheorem{eg}[thm]{Example}

\theoremstyle{definition}
\newtheorem{defn}[thm]{Definition}

\theoremstyle{remark}
\newtheorem{rmk}[thm]{Remark}

\begin{document}
\title[On generalized Hilbert--Kunz function and multiplicity]
{On generalized Hilbert--Kunz function and multiplicity}
\author{Hailong Dao}
\address{Department of Mathematics\\
University of Kansas\\
 Lawrence, KS 66045-7523 USA}
\email{hdao@ku.edu}

\author{Ilya Smirnov}
\address{Department of Mathematics\\
University of Virginia\\
 Charlottesville, VA 22904-4137 USA}
\curraddr{
Department of Mathematics\\
Stockholm University\\
SE - 106 91 Stockholm, Sweden
}
\email{smirnov@math.su.se}

\date{\today}
\thanks{The first author is partially supported by NSF grant 1104017. Part of this  work is  supported by the National Science Foundation under Grant No. 0932078 000, while the authors were in residence at the Mathematical Science Research Institute (MSRI) in Berkeley, California, during the Commutative Algebra year in 2012-2013.}


\numberwithin{thm}{section}
\numberwithin{equation}{section}
\begin{abstract}
Let $(R,\m)$ be a local ring of characteristic $p>0$ and $M$ a finitely generated $R$-module. In this note we consider the limit: $\lim\limits_{n\to \infty} \frac{\length(\lc^0_\m(F^n(M)))}{p^{n\dim R}} $ where $F(-)$ is the Peskine--Szpiro functor. 
A consequence of our main results shows that the limit always exists when $R$ is excellent, equidimensional and has an isolated singularity. Furthermore, if $R$ is a complete intersection, then the limit is $0$ if and only if the projective dimension of $M$ is less than the Krull dimension of $R$.  We exploit this fact to give a quick proof that if $R$ is a complete intersection of dimension $3$, then the Picard group of the punctured spectrum of $R$ is torsion-free. Our results work quite generally for other homological functors and can be used to prove that certain limits recently  studied by Brenner exist over projective varieties. 
\end{abstract}

\maketitle

\section{Introduction}\label{intro}

For local rings in positive characteristic there is an intrinsic theory of multiplicity derived from the Frobenius endomorphism.
This notion originates from the work of Kunz (\cite{Kunz}) and was defined by Monsky (\cite{Monsky}) and called {\it Hilbert--Kunz} multiplicity.
Recently this theory have been intensely studied due to connections to tight closure theory, birational geometry, and its inherent very interesting and mysterious behavior. For a recent survey we refer the interested  to  (\cite{Hu}).

Let $(R, \m)$ be a local ring of characteristic $p > 0$ and dimension $d$ and $M$ be a finitely generated $R$-module. 
If $I$ is an $\m$-primary ideal, then $I^{[p^n]} := \{x^{p^n} \mid x \in I \}$ is also $\m$-primary. 
We define the {\it Hilbert--Kunz multiplicity} of $I$ as the limit
\[
\ehk(I) = \lim_{n \to \infty} \frac{\length (R/I^{[p^n]})}{p^{nd}}.
\]
The existence of this limit is not trivial and is a result of Monsky (\cite{Monsky}).
The sequence of lengths $\length (R/I^{[p^n]})$ is often called the Hilbert--Kunz function of $I$.

The definition can be naturally extended to finite length modules.
Let $\ps R$ denote an $R$-algebra obtained by the $n$th iterate of the Frobenius endomorphism, 
{\it i.e.,} $\ps R$ is abstractly isomorphic to $R$  as a ring and $r\ps s = \ps (r^{p^n}s)$ for any elements $r\in R$ and $\ps s \in \ps R$. 
Then, we may call the {\it Hilbert--Kunz multiplicity} of a finite length module $N$ the limit
\[
\ehk(N) = \lim_{n \to \infty} \frac{\length_{\ps R} (N \otimes_R \ps R)}{p^{nd}}.
\]
The existence of the limit was shown by Seibert (see Theorem~\ref{sei}).
It is easy to see that we indeed recover the original definition from the finite length module $N = R/I$.
 
In this paper we propose a further extension to arbitrary finitely generated modules, and, thus, not necessary $\m$-primary ideals.
Namely, for a finitely generated $R$-module $M$ we study the function  
\[\fhk M(n) :=\length_{\ps R}(\lc^0_\m(M \otimes_R \ps R))\]
and the limit (if it exists!)
\[\hk(M) := \lim_{n \to \infty} \frac{\fhk M(n)}{p^{nd}},\]
which we respectively call the {\it generalized Hilbert--Kunz function} and {\it generalized Hilbert--Kunz multiplicity} of $M$. 
Indeed, it is still not hard to see that if $I$ is an $\m$-primary ideal, then 
$\fhk {R/I}(n)$ is the usual Hilbert--Kunz function. 

As far as we know, 
$\fhk M$ appears for the first time in a paper by Ian Aberbach \cite{Ab} where it was shown that if $R$ is a domain essentially of finite type over a field and $\dim R/I=1$, then $\frac{\fhk {R/I}(n)}{p^{nd}}$ is bounded from above.  Recently, a more general form of this definition was studied by Epstein--Yao in \cite{EY} (see Definition \ref{keydef}).

However, the main inspiration for our work is 
is a recent astonishing result by Cutkosky (\cite[Corollary~11.3]{Cut}), 
who showed that under very mild conditions, the limit:
\[\lim_{n \to \infty} \frac{\length(\lc^0_{\m}(R/I^n))}{n^d}\]
exists. This limit is called $\epsilon$-multiplicity and was first defined as a limsup by Katz and Validashti in \cite{JK}.

The following theorem presents two important special cases 
of our main existence result, Corollary~\ref{highlc}, and the characterization of positivity from Corollary~\ref{posci}

\begin{thm}
Suppose that $M_\p$ has finite projective dimension for every prime ideal $\p \neq \m$ (this always holds when $R$ has an isolated singularity). Then 
\begin{enumerate}
\item If $R$ is excellent, equidimensional, and locally Cohen-Macaulay on the punctured spectrum, then $\hk(M)$ exists. 
\item If $R$ is Cohen-Macaulay, then $\hk(M)$ exists.
\end{enumerate}

Moreover, if $R$ is a complete intersection, then $\hk(M)=0$ if and only if $\pd M<\dim R$.  
\end{thm}

The last assertion is related to  two different facts about  complete intersections: 
Dutta and Miller have proved (\cite{D1, M1}) for a finite length module $M$ that $\hk (M) = \length(M)$ if and only if $\pd M<\infty$; 
and it was shown in~\cite{DLM} that under the assumptions of the Theorem, $\fhk M(n)=0$ for some $n$ if and only if 
$\pd M<\dim R$.  In Example~\ref{proj dim ex} we show necessity of the assumptions.

Our methods give more general results on the asymptotic behavior of various functors on the iterations of the Frobenius endomorphism, see  
\ref{Torci}, \ref{homcplx}, \ref{Extci}, \ref{Extci}. Results in the same spirit have been obtained by a number of authors  (\cite{AL, AM, D2, Li, M2}). 
For instance, we are able to prove that limits exist for higher local cohomology modules (see Theorem \ref{extlimit}). A special case of that result yields: 

\begin{cor}
Suppose that either $(R, \m)$ is Cohen-Macaulay or excellent and locally Cohen-Macaulay on the punctured spectrum. 
If $M$ a finitely generated $R$-module such that 
$\pd_{R_\p} M_\p < d - k$ on the punctured spectrum, then
\[\lim_{n \to \infty} \frac{\length_R (\lc^k_{\m} (F^n(M)))}{p^{nd}}\]
exists.

In particular, if $M$ is locally free on the punctured spectrum, then
the limit above exists for all $k < d$.
\end{cor}

These higher Hilbert--Kunz functions recently appeared in the work of Brenner, who
used our results, Theorems \ref{refl} and \ref{pre1}, to show 
that there exists a finite length module over a local hypersurface  with the irrational Hilbert--Kunz multiplicity (\cite{Bre}).
Furthermore, from this he is able to deduce that the Hilbert--Kunz multiplicity of a local ring may be irrational, 
thus providing a negative answer to a long-standing question.

We also apply our techniques to prove existence of the limits studied by Brenner over projective varieties in \cite{Bre}. In the following $F^{n*}$ denote the $n$th-iteration Frobenius pull-back. 
\begin{cor}
Let $X$ be a polarized projective variety over a field $k$ of characteristic $p$ of dimension $d$, with a fixed very ample invertible sheaf $\mathcal O_X(1)$. Let $\mathcal F$ be a vector bundle on $X$. Then for each $0<k<\dim X$ the limit
\[ \lim_{n\to\infty} \frac{\sum_{m\in \mathbb Z} h^k((F^{n\ast}\mathcal F)(m))}{p^{n(d+1)}}\]
exists if $X$ is  $(S_{k+2})$.
\end{cor}

As another application we strengthen the main result of \cite{DLM} and reprove that
the Picard group of the punctured spectrum of a complete intersection has no torsion elements (see Theorem \ref{refl}).

\begin{thm}
Let $R$ be a local ring satisfying Serre's condition $(S_2)$ and  $\dim R\geq 3$. Let $I$ be a reflexive ideal that is locally free on $\Spec R-\{\m\}$. Then 
\begin{enumerate}
\item $\lim\limits_{n \to \infty} \frac{\length(\lc^2_{\m}(I^{(p^n)}))}{p^{nd}} = 
\lim\limits_{n \to \infty}  \frac{\length(\lc^2_{\m}(I^{\otimes p^n}))}{p^{nd}}$ exists. 
\item When $R$ is a complete intersection, the limit in part $(1)$ is $0$ if and only if $I$ is principal. In particular, the Picard group of  $\Spec R-\{\m\}$ has no torsion elements.
\end{enumerate}
\end{thm}

Other aspects of the theory, such as the connections to tight closure theory and $F$-singularity of pairs, and  some computations of generalized Hilbert--Kunz limits, will be dealt with in forthcoming works (\cite{DS}, \cite{DW}). 
We also want to mention recent  papers that deal with some issues raised in the preprint version of this work. For example, \cite{BreCa} computes $\hk(M)$ over a $2$-dimensional graded normal domain using geometric techniques, while \cite{Vra} generalizes Proposition~\ref{prop:LC}.

\section{Preliminaries}\label{prelim}

The Frobenius endomorphism provides two important functors for modules over a commutative ring $R$.
Let $M$ be an $R$-module. 
First, we may consider the restriction of scalars from the iterated Frobenius map which is denoted as $\ps M$. 
The most important instance is the algebra $\ps R$, which, if $R$ is reduced, 
can be identified with the ring of $p^n$-roots $R^{1/p^n}$.
For any $R$-module $M$, $\ps M$ is naturally an $\ps R$-module, $\ps r \cdot \ps m = \ps rm$.
Being a restriction of scalars, $\ps M$ is an exact functor.

The second functor originates from  the base change along
the Frobenius endomorphism: $F_R^n (M)$ is an $R$-module such that
$\ps (F_R^n (M)) = M\otimes_R \ps R$ as $\ps R$-modules.
This functor is called the {\it Peskine--Szpiro functor} and was introduced by Peskine and Szpiro in \cite{PS2}.
The Peskine--Szpiro functor is right-exact and
the values of the derived functors $\Tor^R_i (M, \ps R)$ are similarly viewed as $R$-modules via the target of the base change map. 
Note that $F^n(R) \cong R$ and for cyclic modules
$F^n(R/I) \cong R/I^{[p^n]}$.

We will use $\length(M)$ and $\pd M$ to denote the
\emph{length} and \emph{projective dimension}, respectively,  of the module
$M$. We will slightly abuse notation and denote $\dim M = \dim \Supp M$. 
We use the notation $\ul x$ for a sequence of elements of $R$. We say that a ring $R$ is $F$-finite, if $\ps R$ is a finitely generated $R$-module for any (equivalently, all) $n$.

If $f,g \colon \mathbb{N} \to \mathbb{R}$ are sequences, we say that
$f = O(g)$ if there exists a constant $C$ such that $f(n) \leq Cg(n)$ for all $n$.
We say that $f = o(g)$ if $\lim\limits_{n\to \infty} f(n)/g(n) = 0$.


\begin{defn}\label{IPD} Let $M$ be an $R$-module. One defines the {\it infinite projective dimension locus} of $M$ as
\[\IPD(M)=\{\p \in \Spec R \mid \pd_{R_{\p}} M_{\p}=\8\}.\]

\end{defn}

The category of  modules $M$ with $\IPD(M) \inc \{\m\}$, the modules that have the finite projective dimension on the punctured spectrum,  contains all modules of finite lengths, modules of finite projective dimensions and is closed under taking extension, direct summand, kernel of epimorphism and cokernel of monomorphism.  In particular, when $R$ has an isolated singularity, all finitely generated modules satisfy this property.  

Now, we can introduce our main object of study.

\begin{defn}\label{keydef}
Let $M$ be a finitely generated module over $R$. One defines
\[\eu(M) = \limsup_{n \to \infty} \frac{\length(\lc^0_\m(F^n(M)))}{p^{n\dim R}}\]
and
\[\ed(M) = \liminf_{n \to \infty} \frac{\length(\lc^0_\m(F^n(M)))}{p^{n\dim R}}.\]
If $\eu(M)=\ed(M)$ we denote the same value by $\hk(M)$. We call $\fhk M(n) :=\length(\lc^0_\m(F^n(M))$ 
the {\it generalized Hilbert--Kunz function} of $M$.
\end{defn}
This can be seen as a special case of relative multiplicity defined by Epstein and Yao in~\cite{EY}.
In their notation $\eu(M) = u_M^+(0,M)$ and $\ed(M) = u_M^-(0,M)$.

\begin{rmk}\label{rPeskine-Szpiro}
Since $\ps M$ is an exact functor, $\ps (R/\m)$ is the residue field of $\ps R$ and
$\length_{\ps R} (\ps M) = \length_R (M)$ for any finite length module $M$.
Furthermore, using the definition of the local cohomology via the \v{C}ech complex
one may confirm that 
$\ps \lc^0_\m (M) \cong \lc^0_\m (\ps M)$ for any finitely generated module $M$.  
Combining these facts together we see that
\[
\length_R (\lc^0_\m(F^n(M))) = \length_{\ps R} (\lc^0_\m(M \otimes_R \ps R)),
\]
thus Definition~\ref{keydef} is equivalent to the definition given in Introduction.
\end{rmk}

In \cite{Monsky} Monsky defined and proved existence of Hilbert--Kunz multiplicity.
His method can be used for a variety of limits, the following theorem of Seibert (\cite {Seibert}) can be considered as the most general result that follows by Monsky's proof.

\begin{thm}\label{sei}
Let $(R, \m)$ be a local $F$-finite ring of characteristic $p > 0$ with a perfect residue field. 
Consider a family $\mathfrak C$ of finitely generated $R$-modules such that 
for any short exact sequence
\[\ses{M'}{M}{M''}\]
$M \in \mathfrak C$ if and only if $M', M'' \in \mathfrak C$.
Let $g$ be a function $\mathfrak C \to \mathbb R$ 
such that 
\begin{enumerate}
\item $g(M) \leq g(M') + g(M'')$ for any short exact sequence as above,
\item $g(M \oplus N) = g(M) + g(N)$ for any $M, N \in \mathfrak C$.
\end{enumerate}
Suppose that all modules in $\mathfrak C$ have dimension at most $j$, then
for any $M \in \mathfrak C$ there is a constant $c(M)$ such that  
\[g(\ps M) = c(M)q^j + O(q^{j-1}).\]
\end{thm}

We will need a full version to deal with a spectral sequence argument in the next section.
However, it is often enough to consider the family of all finitely generated $R$-modules $\mo(R)$.

\begin{eg}\label{seiexample}
Let $i\geq 0$ be an integer and $N$ be a module such that $\Tor_i(N,X)$ (respectively $\Ext^i(N,X)$) has finite length for all $X\in \mo(R)$. Then $g(M) = \length(\Tor_i(N,M))$ (respectively $g(M) = \length(\Ext^i(N,M))$) satisfies the conditions of Theorem~\ref{sei}.
\end{eg}

We will also need the following celebrated result of Peskine and Szpiro 
(\cite[Th\'eor\`eme~1.7, Th\'eor\`eme~1.13]{PS2}). 

\begin{thm}\label{tPeskine-Szpiro}
Let $R$ be a Noetherian ring of characteristic $p > 0$.
Let $M$ be a finitely generated $R$-module of finite projective dimension. 
Then $\Tor_i^R (M, \ps R) = 0$ for all $n > 0$ and $i > 0$. Thus
$F^n(M)$ is a module of finite projective dimension, and for every prime ideal $\p$ of $R$
$\pd_{R_\p} (F(M))_\p = \pd_{R_\p} M_\p.$
In particular, if
\[
0 \to F_s \xrightarrow{\phi_s} F_{s-1} \to \ldots \to F_1 \xrightarrow{\phi_1} F_0
\]
is an exact complex of free modules then the complex
\[
0 \to F_s \xrightarrow{\phi_s^{[p^n]}} F_{s-1} \to \ldots \to F_1 \xrightarrow{\phi_1^{[p^n]}} F_0
\]
is exact, where the maps $\phi_i^{[p^n]}$ are obtained by raising the entries of the matrix defining 
$\phi_i$ to power $p^n$.
\end{thm}

The following immediate corollary will be used extensively. 

\begin{cor}\label{PScor}
Let $(R, \m)$ be a local ring of characteristic $p > 0$ and $M$ be a finitely generated $R$-module.
If $\IPD(M) \subseteq \{\m\}$, then $\Tor_i^R (M, \ps R)$ has finite length for all $i, n > 0$.
\end{cor}

We want to finish with a technical observation that will help us to reduce to a case where Theorem~\ref{sei} can be applied.
Namely, the construction allows us to pass, without changing the sequence $\fhk n(M)$, 
to a complete ring with a perfect residue field, which is F-finite by Corollary~2.6 of \cite{Kunz}.

\begin{prop}\label{complete prop}
Let $(R, \m, k)$ be a local ring of positive characteristic $p$ and $S = \hat R \cotimes k^{\infty}$, where $k^\infty$ is the perfect closure of a coefficient field $k$ of $\hat R$.
Then $S$ is a complete F-finite faithfully flat local $R$-algebra with a perfect residue field and the maximal ideal $\m S$.
Moreover, this extension has the following properties:
\begin{enumerate}
\item $\length_R(\lc^i_{\m} (F^n_R(M))) = \length_S(\lc^i_{\m S} (F^n_S(M \otimes_R S)))$;
\item if $\IPD(M) \subseteq \{\m\}$, then $\IPD(M \otimes_R S) \subseteq \{\m S\}$;
\item if $R$ is excellent and satisfies $(S_n)$ on the punctured spectrum (or on the whole spectrum)
$S$ will satisfy the same property;
\item if $R$ is excellent and equidimensional, $S$ is also equidimensional.
\end{enumerate}
\end{prop}
\begin{proof}
Because completion is flat and the the tensor product is taken over a field, $S$ is flat $R$-algebra. 
Moreover, $S$ is local as the complete tensor product of local $k$-algebras and its residue field 
can be computed as 
$
S/\m S \cong k^\infty \cotimes R \otimes_R R/\m \cong k^\infty \cotimes k \cong k^\infty.
$

Due to flatness of $S$ for any $R$-module $M$ we have an isomorphism
$\lc^i_{\m} (M) \otimes_R S \cong \lc^i_{\m S} (M\otimes_R S)$.
Moreover, one also sees that
\[F^n_R(M) \otimes_R S = M \otimes_R \ps R \otimes_R S \cong M \otimes_R \ps S = 
(M \otimes_R S) \otimes_S \ps S = F^n_S (M \otimes_R S).\]
Therefore $\lc^i_{\m} (F^n_R(M))\otimes_R S \cong \lc^i_{\m S} (F^n_S(M \otimes_R S))$ as $S$-modules. 
Since $\m S$ is the maximal ideal of $S$ and $S$ is flat,  
$\length_R (K) = \length_S (K \otimes_R S)$ for any finite length $R$-module $K$, so
\[
\length_R \left(\lc^i_{\m} (F^n_R(M)) \right) = \length_S \left(\lc^i_{\m} (F^n_R(M)) \otimes_R S \right)
= \length_S \left (\lc^i_{\m S} (F^n_S(M \otimes_R S)) \right).
\]

Now, let $\mathfrak q$ be a prime ideal of $S$ and let $\p = \mathfrak q \cap R$.
Then $R_\p \to S_{\mathfrak q}$ is faithfully flat, so, if $\pd M_\p < \infty$ and $F$ is a finite free resolution of $M_\p$, 
then $F \otimes_{R_\p} S_{\mathfrak q}$ is a finite free resolution of $M_\p \otimes_{R_\p} S_{\mathfrak q}$.

The third assertion follows from \cite[Theorem 4.5]{AF} by Avramov and Foxby, since 
an excellent ring has regular formal fibers.

Last, we observe that if $R$ is excellent and equidimensional, then $\hat{R}$ is also equidimensional. 
Furthermore, by Cohen's structure theorem $\hat{R}$ is a homomorphic image of a power series ring $A = k[[x_1, \ldots, x_d]]$,
say $\hat{R} = A/I$. In this case, by the construction, $S = B/IB$ for $B = k^\infty[[x_1, \ldots, x_d]]$.  
An old result of Chevalley (\cite[Prop. 9]{Chev}) states that $\dim A/P = \dim B/Q$
for any prime $P$ of $A$ and any minimal prime ideal $Q$ of $PB$.
Now, the claim follows by applying the Chevalley's result to the minimal primes of $I$.
\end{proof}

\section{Existence of limits}

Our first result allows us to pass from $\fhk M(n)$ to a sequence for which we will be able to use
Seibert's theorem. The proof is based on the Grothendieck local duality (\cite[Corollary~6.3]{Hartshorne}) which asserts 
that $\lc^i_{\m}(M) \cong \Hom_R (\CH^{d-i}(\Hom(M,D_R)), E(k))$,
where $R$ is a local ring, 
$D_R$ is the normalized dualizing complex of a local ring $R$, $E(k)$ is the injective hull of the residue field,  
and $M$ is a finitely generated module.
We should remark that a complete ring always admits a dualizing complex, since it is an image of a regular local ring.
We will also use that if $R$ is equidimensional and catenary, then the dualizing complex localizes, so
$(\CH^n(D_R))_\p \cong \CH^{n}(D_{R_\p}).$
We refer the reader to \cite{Schenzel} for the background on dualizing modules
and \cite{Hartshorne, Weibel} for the background on derived functors. 

The idea of the proof originates from the Cohen-Macaulay case, 
where the assumption $\IPD(M) \subseteq \{\m\}$ gives that
$
\Ext^d_R (F^n(M), \omega_R) \cong \Ext^d_R (M, \ps \omega_R),
$
and the latter comes from a functor suitable for Seibert's theorem.

Before we start the proof, we record a technical lemma based on the derived Hom-tensor adjunction.

\begin{lemma}\label{derived lemma}
Let $(R, \m)$ be a $F$-finite local ring of dimension $d$ with the dualizing complex $D_R$.
Let $M$ be a finitely generated $R$-module such that $\IPD (M) \subseteq \{\m\}$
and let $F$ be its free resolution.
Then for every $0 \leq k \leq d$ we have the isomorphism
\[
\CH^{k} \big( \Hom_R (M \otimes_R \ps R, D_R) \big) 
\cong \CH^{k}(F^*\otimes_R D_{\ps R})\big.
\]
\end{lemma}
\begin{proof}
Let $S = \ps R$.  We are going to compute $\CH \big (\rh(M \dtor_R S, D_R)\big )$ via 
the hypercohomology spectral sequence of a double complex (\cite[5.6]{Weibel}):
\[
E^{pq}_2 = \CH^p \left( \Hom_R (\CH^{-q}(M\dtor_R S), D_R) \right) 
\Longrightarrow  \CH^{p+q} \left( \rh (M\dtor_R S, D_R) \right).
\]
Because $M$ has finite projective dimension on the punctured spectrum, 
Corollary~\ref{PScor} shows that $\CH^{-q}(M\dtor_R S) = \Tor_q^R (M, S)$ has finite length for any $q > 0$
and  by the local duality
\[
\CH^{p} \big( \Hom_R (\CH^{-q}(M\dtor_R S), D_R) \big) \cong  \Hom_R  \left(\lc^{d-p}_{\m} \left (\Tor_{q}^R(M, S) \right), E(k) \right)= 0.
\]
Hence the spectral sequence collapses and shows the isomorphism
\[
\CH^{k} \big( \Hom_R (M\otimes_R S, D_R) \big) \cong 
\CH^{k} \big( \rh(M\dtor_R S, D_R) \big).
\]
Since $S$ is finite over $R$, by the derived Hom-tensor adjunction (\cite[Proposition~5.15]{Hartshorne})
\[
\rh(M\dtor_R S, D_R) \cong \rh(M,\rh(S,D_R)) \cong \rh(M,D_S).
\] 
Now, by comparing complexes $\Hom_R (F, D_S)$ and $F^* \otimes_R D_S$ one can get 
the isomorphism
\[
\CH^{k} (\rh(M,D_S)) \cong \CH^{k}(\Hom_R(F, D_S)) \cong \CH^{k}(F^*\otimes_R D_S).
\]
Thus, after combining the steps, we have proved that
\[
\CH^{k} \big( \rh(M\dtor_R S, D_R) \big) \cong \CH^{k} (\rh(M,D_S))
\cong \CH^{k}(F^*\otimes_R D_S).
\]
\end{proof}

Another auxiliary lemma uses Seibert's theorem to analyze limits involving the homology of the dualizing complex.
For convenience, the depth of the zero module is taken to be $\infty$. 

\begin{lemma}\label{limits are zero}
Let $(R, \m)$ be an equidimensional $F$-finite local ring of dimension $d > 0$
and $M$ be a finitely generated $R$-module such that $\IPD(M) \inc \{\m\}$.
For an integer $k$,  $0 \leq k < d$ suppose that $R$ satisfies $(S_{k + 1})$ on the punctured spectrum
and for all prime ideals $\p \neq \m$ at least one of the following conditions holds:
\begin{enumerate}
\item $R_\p$ is Cohen-Macaulay, 
\item $M_\p$ is free, 
\item or $\depth M_\p \geq \dim R_\p - d + k + 2$.
\end{enumerate}
Then for any integers $j > 0, i \geq 0$ such that $i + j\geq d - k - 1$
\[
\lim_{n \to \infty} \frac{\length (\Ext_R^{i}(M, \ps \CH^j(D_R)))}{p^{nd}} = 0.
\]  
\end{lemma}
\begin{proof}
Let $C_i = \{ X \in \mo(R) \mid \Supp X \subseteq \Supp \CH^{i}(D_R) \}$.
We claim that the assumptions imposed on the primes $\p \neq \m$ guarantee that 
$\Ext^{j}_R (M, X)$ has finite length for all $i, j > 0$ such that $i + j \geq d-k-1$ and all $X \in C_i$.
Namely, we will show the following.

\begin{claim}
For every prime $\p \neq \m$ either $\CH^{i}(D_{R})_\p = 0$ or $\pd M_\p < j$.
\end{claim}
\begin{proof}[Proof of the claim]
The claim is clear if we assume that $M_\p$ is free.

Since $R$ is equidimensional, $\CH^{i}(D_{R})_\p = \CH^{i}(D_{R_\p})$, 
and, by local duality, $\CH^{i}(D_{R_\p})$ is dual to $\lc^{\dim R_{\p}-i}_\p (R_\p)$. 
Hence $\CH^{i}(D_{R})_\p = 0$ whenever $\depth R_\p > \dim R_\p - i$.
In particular, if $R_\p$ is Cohen-Macaulay, then $\CH^{i}(D_{R_\p}) = 0$ for all $i > 0$ and all $j$.

Last, assume that $\depth M_\p \geq \dim R_\p - d + k + 2$. 
Then by the Auslander--Buchsbaum formula
\[
\pd M_\p = \depth R_\p - \depth M_\p \leq \depth R_\p - \dim R_\p + d - k - 2.  
\]
Therefore, if $j  \leq  \pd M_\p$, then 
\[
i \geq d - k - 1 - j \geq d - k - 1 - (\depth R_\p - \dim R_\p + d - k - 2) = \dim R_\p - \depth R_\p + 1.
\]
And, as explained above, this implies that $\CH^{i}(D_{R})_\p = 0$.
\end{proof}

If  $i,j > 0$ and $i + j \geq d - k -1$, then the claim shows that we may apply Theorem~\ref{sei} to the family $C_i$ and the function $g(X) = \length (\Ext^j_R (M, X))$. 
Since $\dim \Supp \CH^i(D_R) \leq d - i$, Seibert's theorem shows that
\[\lim_{q \to \infty} 
\frac{\length \left(\Ext^j_R (M, \ps \CH^{i}(D_R))\right)}{p^{nd}} = 0
\]
for any $i> 0$. 

Now, it is left to analyze the modules $\Hom_R (M, \ps \CH^{d-k}(D_R))$.
Since $R$ satisfies $(S_{k+1})$ on the punctured spectrum, 
$\lc^{m}_\p (R_\p) = 0$ for all $m \leq \min (k, \dim R_\p - 1)$, so
$\CH^{j}(D_R)_\p = 0$ if $j = \dim R_\p - m \geq \max (1, \dim R_\p - k)$.
Therefore, if $j \geq \max (d- k-1, 1)$, then 
$\CH^{j}(D_R)$ has finite length
and we may apply Seibert's theorem to the function $\Hom (M, X)$ on $C_j$. 
It follows from the theorem that for all such $j$
\[ 
\lim_{q \to \infty} 
\frac{\length \left(\Hom_R (M, \ps \CH^{j}(D_R))\right)}{p^{nd}} = 0. 
\] 
\end{proof}

\begin{thm}\label{extlimit}
In the assumptions of Lemma~\ref{limits are zero}
\[
\lim_{n \to \infty} \frac{1}{p^{nd}}
\left(\length \left(\lc^k_{\m}(F^n(M))\right)- \length \left (\Ext_R^{d-k}(M, \ps \CH^0(D_R)) \right)\right) = 0.
\]  
In particular, if $\pd_{R_\p} M_\p < d - k$ on the punctured spectrum, then 
$\lim\limits_{n \to \infty} \frac{\length (\lc^k_{\m}(F^n(M)))}{p^{nd}}$ exists.
\end{thm}
\begin{proof}
For the second assertion, we note that the additional assumption on $M$ implies that 
$\Ext^{d-k}_R (M, X)$ has finite length for any finite module $X$.
Hence 
\[\lim_{n \to \infty} \frac{\length_R \left(\Ext^{d-k}_R (M, \ps \CH^0(D_R) ) \right)}{p^{nd}}\]
exists by Theorem~\ref{sei} and the second assertion follows from the first.

By \cite[Theorem~2.5]{Kunz} an F-finite ring is excellent,
so by Proposition~\ref{complete prop} we may assume that $R$ is complete and has a perfect residue field. 
By the local duality and Lemma~\ref{derived lemma}, we have 
\[
\lc^k_{\m}(M\otimes_R \ps R) \cong \big( \CH^{d-k} \big( \Hom_R(M \otimes_R \ps R, D_R) \big) \big)^\vee 
\cong \big (\CH^{d-k}(F^*\otimes_R D_{\ps R})\big )^\vee.
\]
Now, we will compute $\length \left (\CH^{d-k}(F^*\otimes_R D_{\ps R}) \right )$ 
using the standard spectral sequence 
\[E^{pq}_2 = \CH^{p}(F^*\otimes_R \CH^q(D_{\ps R})) \Longrightarrow  \CH^{p+q}(F^*\otimes_R D_{\ps R}).\]
Since $D_{\ps R} = \ps D_R$  and since the restriction of scalars is exact, 
$\CH^{q}(D_{\ps R}) = \ps \CH^{q}(D_R)$. Hence $E^{p,q}_2 = \Ext^p_R (M, \ps \CH^q (D_R))$.


We claim that for any $0 < i \leq d-k$
the contribution of $E^{d-k-i, i}_\8$ to $\CH^{d-k}(F^*\otimes D_{\ps R})$ is $o(p^{nd})$. 
First, observe that 
the entry $E^{d-k-i, i}_{s+1}$ is a quotient of a submodule of $E^{d-k-i, i}_s$, 
so 
$\length(E^{d-k-i, i}_{s+1}) \leq \length(E^{d-k-i, i}_{s}) \leq \length(E^{d-k-i, i}_2)$.
But by Lemma~\ref{limits are zero}
\[\lim_{q \to \infty} \frac{\length(E^{d-k-i, i}_2)}{p^{nd}} = 
\lim_{q \to \infty} \frac{\length(\Ext_R^{d-k-i} (M, \ps\CH^{i}(D_R)))}{p^{nd}} = 0.\]

We also need to estimate the contribution of $E^{d-k,0}_\infty$.
Because for any $s \geq 2$ the map on the $s$th sheet from $E^{d-k,0}_s$ is zero,
$E^{d-k,0}_{s+1}$ is the cokernel of $E^{d-k - s,s - 1}_{s} \to E^{d-k,0}_{s}$.
Thus
\[\length(E^{d-k,0}_{s}) \geq 
\length(E^{d-k,0}_{s+1}) \geq \length(E^{d-k,0}_{s}) - \length(E^{d-k - s,s - 1}_{s})
\geq \length(E^{d-k,0}_{s}) - \length(E^{d-k - s,s - 1}_{2}).\]
Since our spectral sequence converges after $d - k$ steps, repeatedly applying the inequalities above we obtain  that
\[
\length(E^{d-k,0}_{2}) \geq 
\length(E^{d-k,0}_{\infty}) \geq   
\length(E^{d-k,0}_{2}) - \sum_{s = 2}^{d-k} \length(E^{d-k - s,s - 1}_{2}), \text { so }
\]
\[
\sum_{s = 2}^{d-k} \length \left(\Ext^{d-k - s}_R (M, \ps \CH^{s - 1}(D_R))\right) \geq 
\length \left(\Ext^{d - k}_R(M, \ps \CH^0(D_R))\right) - \length\left(E^{d-k,0}_{\infty}\right) \geq 0.
\]
However, $(d -k - s) + (s - 1) = d - k - 1$, so by Lemma~\ref{limits are zero}
the left-hand side is $o(p^{nd})$.
Therefore, 
\[
\lim_{n \to \infty} \frac{1}{p^{nd}}
\left( \length \left (\CH^{d-k}(F^*\otimes_R D_{\ps R}) \right ) - \length \left (\Ext_R^{d-k}(M, \ps \CH^0(D_R)) \right) \right ) = 0.
\]  
Now the theorem follows, since by Remark~\ref{rPeskine-Szpiro} if the residue field is perfect
\[\length_R (\lc^k_{\m}(M\otimes_R \ps R)) = \length_{\ps R} (\lc^k_{\m}(M\otimes_R \ps R)) = \length_R (\lc^k_{\m} (F^n(M) )).\]
\end{proof}

\begin{cor}\label{highlc}
Let $(R, \m)$ be an equidimensional local ring of dimension $d$.
Suppose that either $R$ is Cohen-Macaulay or excellent and Cohen-Macaulay on the punctured spectrum. 
Let $M$ be a finitely generated $R$-module and $0 \leq k < d$ be an integer  such that
$\pd_{R_\p} M_\p < d - k$ on the punctured spectrum, then
\[\lim_{n \to \infty} \frac{\length_R \left(\lc^k_{\m} (F^n(M))\right)}{p^{nd}}\]
exists.

In particular, if $M$ is locally free on the punctured spectrum,
the limit above exists for all $k < d$.
\end{cor}
\begin{proof}
The Artinian case follows from Seibert's theorem, so assume $d > 0$.
Using the construction in Proposition~\ref{complete prop}, we may assume that $R$ is $F$-finite without affecting relevant issues,
so we can apply the theorem.
\end{proof}

When we are interested in $\lc^0_\m(F^n(M))$, we may relax the assumptions. 

\begin{cor}\label{34}
Let $(R, \m)$ be an equidimensional excellent local ring that satisfies $(S_1)$ on the punctured spectrum and 
$M$ be a finite $R$-module such that $\IPD (M) \inc \{\m\}$.
Assume that for any prime $\p \neq \m$   
either $R_\p$ is Cohen-Macaulay or $M_{\p}$ is free.
Then $\hk(M)$ exists. 

In particular, $\hk(M)$ exists if either 
\begin{enumerate}
\item $\dim M \leq 1$, or 
\item $M$ is locally free on the punctured spectrum.
\end{enumerate}
\end{cor}
\begin{proof}
The first part follows from Theorem~\ref{extlimit}.
If $\dim M = 1$, then for any prime ideal $\p$ of dimension one, 
$M_\p$ is a finite length module of finite projective dimension, so 
the Peskine--Szpiro Intersection Theorem (\cite{PS1}) implies that $R_\p$ is Cohen-Macaulay.
\end{proof}

Because we only need to control the ($d-k$)th diagonal of the spectral sequence, 
an upper bound exists in a more general situation.

\begin{thm}
Let $(R, \m)$ be an equidimensional $F$-finite local ring of dimension $d > 0$
and $M$ be a finitely generated $R$-module such that $\IPD(M) \inc \{\m\}$.
Let $k \geq 0$ be an integer and 
assume that for all prime ideals $\p \neq \m$  one of the following holds:
$R_\p$ is Cohen-Macaulay, $M_\p$ is free, or $\depth M_\p \geq \dim R_\p - d + k + 1$.
Then, if $R$ satisfies $(S_{k})$ and $\pd M_\p < d-k$ on the punctured spectrum,
\[\length \left(\lc^k_{\m}(F^n(M))\right) = O(p^{nd}).\]
\end{thm}
\begin{proof}
We follow the proof of Theorem~\ref{extlimit} for $k - 1$. 
Lemma~\ref{limits are zero} is still applicable, and we only 
need to estimate the contribution of $E^{d-k,0}_\infty$. 
Since 
\[
\length \left(\Ext^{d - k}_R(M, \ps \CH^0(D_R))\right) = \length(E^{d-k,0}_{2}) \geq  \length(E^{d-k,0}_{\infty}),
\]
we just estimate that the left-hand side is $O(p^{nd})$ by Theorem~\ref{sei}.
The theorem's assumptions are satisfied because $\pd M_\p < d-k$ on the punctured spectrum, 
so $\Ext^{d-k}_{R_\p} (M_\p, X) = 0$ for all $X$ and all $\p \neq \m$.
\end{proof}

It follows that an upper bound for the generalized Hilbert--Kunz function exists quite generally.

\begin{cor}\label{cor bound}
Let $(R, \m)$ be a formally equidimensional local ring.
Then $\eu(M)$ is finite for any finite $R$-module $M$ such that $\IPD(M) \inc \{\m\}$.
\end{cor}
\begin{proof}
We just note that $\depth M_\p \geq \dim R_\p - d  + 1$ is trivially true on the punctured spectrum. 
Thus there is no need in extra assumptions, in particular, since we do not need to preserve the Cohen-Macaulay locus.
\end{proof}

\subsection{Applications.}
We are able to deduce existence of limits for other functors.

\begin{lemma}\label{torsyz}
Let $(R, \m)$ be a formally equidimensional local ring of positive depth.
Then 
\[\lc^0_\m(\Tor_1^R(M, \ps R)) \cong \lc^0_\m (F^n(\syz M)).\]
\end{lemma}
\begin{proof}
From a long exact sequence for tensor product
\[0 \to \Tor_1^R (M, \ps R) \to F^n(\syz M) \to F^n(F) \to F^n(M) \to 0\]
we obtain the exact sequence of local cohomology:
\[0 \to \lc^0_\m(\Tor_1^R (M, \ps R)) \to \lc^0_\m(F^n(\syz M)) \to \lc^0_\m (F^n(F)).\]
Note that $F^n(F) = F$, so $\lc^0_\m (F^n(F)) = 0$ and the lemma follows. 
\end{proof}

\begin{cor}\label{torcm}
Let $(R, \m)$ be a formally equidimensional local ring of positive depth and 
$M$ be a finitely generated $R$-module such that $\IPD(M) \inc \{\m\}$.
Then $\length(\Tor_i^R(M, \ps R)) = O(p^{nd})$. 

Moreover, if $R$ is Cohen-Macaulay, then
\[\lim_{q \to \infty} \frac{\length (\Tor_i^R(M, \ps R))}{p^{nd}}\]
exists for any $i$ and is equal to $\hk (\syz^i M)$.
\end{cor}
\begin{proof}
This follows from Lemma~\ref{torsyz}, Corollary~\ref{34}, and Corollary~\ref{cor bound}.
Note that $\Tor_i^R(M, \ps R)$ has finite length by Corollary~\ref{PScor}.
\end{proof}

Next we want to discuss applications to projective varieties over a field of positive characteristic. In this situation our results can be used to prove that certain limits recently studied by Brenner in \cite{Bre} exist. 
\begin{cor}
Let $X$ be a polarized projective variety of dimension $d$ over a field $k$ of characteristic $p$, with a fixed very ample invertible sheaf $\mathcal O_X(1)$. Let $\mathcal F$ be a vector bundle on $X$. For any $0<k<\dim X$ 
if $X$ is  $(S_{k+2})$, then
the limit
\[ \lim_{n \to \infty} \frac{\sum_{m\in \mathbb Z} h^k((F^{n\ast}\mathcal F)(m))}{p^{n(d+1)}}\]
exists.
\end{cor} 

\begin{proof}
We embed $X$ into a projective space using $O_X(1)$ and let $R$ be the local ring at the vertex of the coordinate ring of $X$ with respect to the said embedding. One can find, up to shifts, a (non-unique) finitely generated module $R$-module $M$ such that $\widetilde M \cong \mathcal F$. There are well-known isomorphisms 
$\lc_\m^{i+1}(M) \cong \oplus_{i \in \mathbb Z} H^{i}(X, \mathcal F(i))$
which allow us to use Theorem \ref{extlimit}.

\end{proof}

We also want to present an upper bound obtained by a different technique albeit depending on a widely open conjecture. 

\begin{defn}
A module $M$  over a local ring $(R, \m)$ 
satisfies condition (LC) if there exists an integer $l$ such that 
$\m^{lp^n}\lc^0_{\m}(F^n(M))=0$ for all $n$.
\end{defn}

This condition arose from the problem of localization of tight closure (for example, see \cite[Theorem 6]{Katzman}). 
In fact (see the discussion after Corollary~3.2 in \cite{Hu2}), 
if all cyclic modules of $R/I$ satisfy (LC) then weakly $F$-regular implies $F$-regular.

The following proposition was further generalized by Adela Vraciu (\cite{Vra}), who showed that under assumptions of the proposition $\hk (R/I)$ exists for all $I$, and the generalized Hilbert--Kunz function
is a linear combination of the classical Hilbert--Kunz functions of $\m$-primary ideals.
Because our proof is short and applies to modules too, we decided to leave it be.
Another generalization of this result appears in \cite{HeJe}, 
where it was shown that the limit exists by requiring $I$ (only!) satisfy a stronger version 
of (LC). 

\begin{prop}\label{prop:LC}
Let $(R, \m)$ be a local ring that satisfies countable prime avoidance.
If all finitely generated modules over $R$ satisfy (LC) then $\eu(M)$ is finite for any finitely generated module $M$. 
\end{prop}

\begin{proof}
We proceed by induction on $\dim M$, the base case $\dim M=0$ follows from Theorem~\ref{sei}. 
Consider the countable set $S = \cup_n \Ass (F^n(M)) \setminus \{\m\}$.
By countable avoidance one can find $x \in \m^l$ such that $x$ is not in any prime of $S$.  
We know that $x^{p^n}\lc^0_{\m}(F^n(M))=0$ for all ${p^n}$.

We claim that $x$ is a nonzerodivisor on $C := F^n(M)/\lc^0_{\m}(F^n(M))$. 
Clearly, $\depth C >0$, so it suffices to show that $\Ass(C)\subseteq S$. 
Let $\p \in \Ass(C)$. Since $\p$ is not $\m$, one can easily see that $F^n(M)_{\p}\cong C_\p$, so $\p \in S$. 

After tensoring the exact sequence $\ses {\lc^0_{\m}(F^n(M))}{F^n(M)}{C}$ 
with $R/(x^{p^n})$, and using
that $\Tor_1^R(C,R/(x^{p^n}))=0$ and $x^{p^n}\lc^0_{\m}(F^n(M))=0$ we obtain
the exact sequence
\[\ses {\lc^0_{\m}(F^n(M))}{F^n(M)/x^{p^n}F^n(M)}{C/x^{p^n}C}.\]
After taking local cohomology we get an inclusion
$\lc^0_{\m}(F^n(M)) \subseteq \lc^0_{\m}\left (F^n(M)/x^{p^n}F^n(M) \right).$
Moreover, applying the right-exact functor $F^n(-)$ to the exact sequence
\[0 \to M \xrightarrow{x} M \to M/xM \to 0,\]
we see that $F^n(M)/x^{p^n}F^n(M) = F^n(M/xM)$.
Thus, $\lc^0_{\m}(F^n(M))$ can be embedded in $\lc^0_{\m}(F^n(M/xM))$, 
and, since $\dim M/xM<\dim M$, the result follows by induction.
\end{proof}

Countable prime avoidance is a mild condition: it is satisfied if $R$ is complete (\cite{Burch}),
or if the residue field is uncountable.

\section{Positivity}

The Hilbert--Kunz multiplicity of a finite length module is positive,
so it is natural to investigate positivity of the generalized version. 
Surprisingly, it is indeed positive over a complete intersection unless the module has non-maximal projective dimension.

First we establish a special case which will be used in the key result, Corollary~\ref{posci}.
A part of the proof easily follows from the general result in the previous section,
but we need it to get the full statement.

\begin{lemma}\label{finpdbig}
Let $(R, \m)$ be a Gorenstein local ring of dimension $d$ and $M$ be a module of finite projective dimension.
If $\pd_R M \leq k$ and $\pd_{R_\p} M_{\p} < k$ on the punctured spectrum, 
then 
\[\lim_{n \to \infty} \frac{\length(\lc^{d-k}_{\m}(F^n(M)))}{p^{nd}} \text{ exists}.\]
Moreover, the limit is positive if and only if $\pd_R M = k$.
\end{lemma}
\begin{proof} 
If $\pd_R M < k$, then by Theorem~\ref{tPeskine-Szpiro} $\pd F^n(M) < k$ for any $n$.  
Hence $\depth (F^n(M)) > d - k$,  so $\lc^{d - k}_{\m} (F^n(M)) = 0$ for all $n$ and there is nothing to prove. 

Now, assume that $\pd_RM=k$. 
Consider a minimal resolution of $M$:
\[0 \to F_k \xrightarrow{\delta_{d}} F_{k-1} \xrightarrow{\delta_{k-1}}  \cdots \xrightarrow{\delta_1} F_0 \to M \to 0.\]
By Theorem~\ref{tPeskine-Szpiro}, a resolution of $F^n(M)$ would look like:
\[0 \to F_k \xrightarrow{\delta_{k}^{[p^n]}} F_{k-1} \xrightarrow{\delta_{k-1}^{[p^n]}}  
\cdots \xrightarrow{\delta_1^{[p^n]}} F_0 \to F^n(M) \to 0.\]
It follows that $\Ext^k_R(F^n(M),R)$ is the cokernel of the map $F_{k-1} \xrightarrow{{\delta_{k}^*}^{[p^n]}} F_{k}$ where $\delta^*$ represents the transposed matrix of $\delta$. 
Thus we obtain that $\Ext^k_R(F^n(M),R) \cong  \Ext^k_R(M,R)\otimes_R \ps R$, 
and its limit exists by Theorem~\ref{sei}
applied to $g(X) = \length(\Ext^k_R(M,R) \otimes_R X)$. 
Note that $\length(\Ext^k_R(M,R)) < \infty$, since $\pd M_{\p} < k$ on the punctured spectrum.

Since $\Ext^k_R(M,R) \neq 0$, $g(X) \geq \length(R/\m \otimes_R X)$ for any $X$.
In particular, $g(\ps R) \geq \length (R/\m^{[p^n]})$.
By Local duality, $\length(\lc^{d - k}_{\m}(F^n(M)) = \length(\Ext^k_R(F^n(M),R)) = g(\ps R)$.
Since $\ehk(R) > 0$, the generalized multiplicity will be positive too.
\end{proof}

\begin{cor}\label {finpd}
Let $R$ be a Gorenstein local ring and $M$ be a module of finite projective dimension.
Then $\hk(M)$ exists and is positive if and only if $\depth M = 0$.
\end{cor}

\begin{prop}\label{edci}
Le $(R, \m)$ be a Gorenstein local ring. The following are equivalent:
\begin{enumerate}
\item  $\hk(M) > 0$ for all $M$ such that $ \IPD(M) = \{\m\}$,
\item  $\hk(M) > 0$ for all $M$ such that $\IPD(M) = \{\m\}$ and $\dim M = \dim R$,
\item  $\hk(M) > 0$ for all maximal Cohen-Macaulay $M$ such that $ \IPD(M) = \{\m\}$.
\end{enumerate}
Moreover, if $R$ is a complete intersection, then the following condition is also equivalent to the first three:
\begin{enumerate} 
\item[(4)]  $\hk(M) > 0$ for all $M$ such that $\IPD(M) = \{\m\}$ and $\depth M = 0$.
\end{enumerate}
\end{prop}
\begin{proof}
Clearly $(1)$ is the strongest condition, so we need to prove the other implications. We assume that $\dim R>0$, 
otherwise the conditions are trivially equivalent. 

First, if $F$ is free, we claim that $\hk (M \oplus F) = \hk(M)$.
Since $F$ is free, the sequence
\[\ses {F^n(M)}{F^n(M \oplus F)}{F^n(F)}\]
is exact and, thus, 
$0 \to \lc^0_{\m}(F^n(M)) \to \lc^0_{\m}(F^n(M \oplus F)) \to \lc^0_{\m}(F^n(F))$
is also exact.
But $\lc^0_{\m}(F^n(F)) = 0$ since $F^n(F) = F$, so we get that $\lc^0_{\m}(F^n(M)) \cong  \lc^0_{\m}(F^n(M \oplus F))$.
This establishes $(2) \Rightarrow (1)$.

Let us  prove that $(3) \Rightarrow (2)$.
Since $R$ is Gorenstein and $\dim M = \dim R$, we can use the Auslander--Bridger approximation (\cite{ABR}) to get
an exact sequence
\[\ses N{M\oplus F}H,\]
where $N$ is maximal Cohen-Macaulay, $F$ is free, 
and $H$ is a module of finite projective dimension such that $\depth H = \depth M$.
From this exact sequence it is easy to see that
\[\length \left (\lc^0_{\m} (F^n(M)) \right) = \length \left (\lc^0_{\m} (F^n(M\oplus F)) \right) 
\geq \length \left(\lc^0_{\m} (F^n(N)) \right).\] 

Now, we prove $(4) \Rightarrow (1)$.
Let $M$ be any module such that $\IPD(M) = \{\m\}$.
By \cite[Theorem~2.5]{DLM}, $\m \in \Ass F^n(M)$ for all $n > 0$, so $\depth F(M) = 0$.
Therefore 
\[\hk(M) 
= \frac 1{p^{\dim R}}\hk (F(M)) > 0.
\]
Note that $F(M)$ can have finite projective dimension, but the limit is still positive by 
the previous corollary.
\end{proof}

\begin{thm}\label{modreg}
Let $(R, \m)$ be a Cohen-Macaulay local ring and $x$ be a regular element.
Suppose $M$ is a finitely generated $R/(x)$-module such that $ \IPD(M) \inc \{\m\}$. 
Then $\hk_R(M) \leq \hk_{R/(x)}(M)$.
\end{thm}
\begin{proof}
Let $d$ denote the dimension of $R$. 
By Proposition~\ref{complete prop}, we may assume that $R$ is complete,  so it has a canonical module $\omega_R$. 

Because $\omega_{R/xR} \cong \omega_R/x\omega_R$, by Theorem~\ref{extlimit} 
we need to compare limits
\[
\hk_{R/x}(M) = \lim_{n\to \infty} \frac{\length \left (\Ext^{d-1}_R(M, \ps (\omega_R/x\omega_R)) \right)}{p^{n(\dim R-1)}}
\]
and
\[
\hk_R(M) = \lim_{n\to \infty} \frac{\length \left(\Ext^{d}_{R}(M, \ps \omega_R) \right)}{p^{n\dim R}}.
\]
Since $M$ is an $R/(x)$-module and $x$ is regular, 
$\Ext^{d}_R(M, \ps \omega_R) \cong \Ext^{d-1}_{R/(x)}(M, \ps \omega_R \otimes_R R/(x))$.
Moreover, applying $- \otimes_R \ps \omega_R$ to the exact sequence
$0 \to (x) \to R \to R/(x)\to 0,$
one can see that $\ps \omega_R \otimes_R R/(x) \cong \ps (\omega_R/x^{p^n}\omega_R)$.
Therefore, 
\[
\hk_R(M) = \lim_{n\to \infty} \frac{\length\left(\Ext^{d-1}_R(M, \ps (\omega_R/x^{p^n}\omega_R))\right)}{p^{n\dim R}}.
\]

Since $x$ is regular on $\omega_R$, for any $m$ the sequence
\[\ses {\omega_R/x^{m-1}\omega_R}{\omega_R/x^m\omega_R}{\omega_R/x\omega_R}\]
is exact. 
Since $\ps (-)$ is an exact functor, this gives a filtration of $\ps (\omega_R/x^{p^n}\omega_R)$ by copies 
of $\ps (\omega_R/x\omega_R)$. 
Applying $\Hom_R (M, -)$ to the filtration, 
we get the exact sequences 
\[\Ext^{d-1}_{R} \left(M, \ps (\omega_R/x^{(m-1)}\omega_R) \right) \to 
\Ext^{d-1}_R \left(M, \ps (\omega_R/x^m\omega_R) \right)  
\to \Ext^{d-1}_R \left(M, \ps (\omega_R/x\omega_R) \right) \to 0.\]
Therefore there is a sequence of inequalities
\[\length \left( \Ext^{d-1}_R(M, \ps (\omega_R/x^{m}\omega_R))\right) 
\leq \length \left(\Ext^{d-1}_R(M, \ps (\omega_R/x\omega_R))\right) + 
\length \left(\Ext^{d-1}_R(M, \ps (\omega_R/x^{m-1}\omega_R)) \right).\]
Thus $\length \left(\Ext^{d-1}_R(M, \ps (\omega_R/x^{p^n}\omega_R)) \right) 
\leq p^n\length \left(\Ext^{d-1}_R(M, \ps (\omega_R/x\omega_R))\right)$ 
and the assertion follows.
\end{proof}

\begin{cor}\label{posci}
Suppose $(R, \m)$ is a local complete intersection and $M$ such that $ \IPD(M) \inc \{\m\}$. 
Then $\hk(M)>0$ unless $\depth M>0$ and $\pd_R M<\infty$ (i.e., $\pd_R M<\dim R$).
\end{cor}
\begin{proof}
If $\pd_R M < \infty$, the statement follows from Corollary~\ref{finpd},
so we are left to prove that $\hk(M) > 0$ for all modules $M$ such that $\IPD(M) = \{\m\}$.
Thus, by Proposition~\ref{edci}, we can assume that $\depth M = 0$.

Moreover, as explained in Proposition~\ref{complete prop}, we may assume that $R$ is complete. 
Hence $R = S/(\ul x)$ for a regular local ring $R$ and a regular sequence $\ul x = x_1, \ldots, x_t$.

By Corollary~\ref{finpd}, we know that $\hk_S (M) > 0$.
Now, we can apply Theorem~\ref{modreg} to get
\[ 0 < \hk_S (M) \leq \hk_{S/x_1S}(M) \leq \cdots \leq \hk_R(M). \]
\end{proof}

As the following example shows, 
the statement is not true without the condition on $\IPD(M)$. 
Later, in Example~\ref{need gor}, we will also observe that 
it is also not enough to assume that $R$ is Gorenstein.

\begin{eg}\label{proj dim ex}
Let $R=k[[x,y,z]]/(x^2y-z^2)$, with $\charac k=2$. Then $(x,z)^{[p^n]}=(x^{p^n})$, so $\hk(R/(x,z))=0$, but $\pd R/(x,z) =\infty$. 
\end{eg}

\begin{cor}\label{Torci}
Let  $(R, \m)$ be a local complete intersection with isolated singularity. 
For a finitely generated module $M$, the following are equivalent:
\begin{enumerate}
\item $\fhk M(n) =0$ for all $n$ (i.e., $\depth F^n (M) > 0$ for all $n$),
\item $\hk(M) =0$,
\item $\pd_R M <\dim R$.
\end{enumerate}
\end{cor}
\begin{proof}
(1) $\Rightarrow$ (2) is trivial, (2) $\Leftrightarrow$ (3) is Corollary~\ref{posci}.
If $\pd_R M < \infty$, $\pd_R F(M) = \pd_R M$
by Theorem~\ref{tPeskine-Szpiro}, and (3) $\Rightarrow$ (1) follows.
\end{proof}

The next corollary gives an asymptotic version of rigidity of $\ps R$ over complete intersection (\cite{AM}) 
for a particular class of modules.

\begin{cor}\label{hyprigid}
Suppose $(R, \m)$ is a local complete intersection and  $M$ is such that $ \IPD(M) \inc \{\m\}$. Then for any $i > 0$
\[\lim_{n\to \infty} \frac{\length \left (\Tor_i^R(M,\ps R) \right)}{p^{n\dim R}}\]
exists and is $0$ if and only if  $\pd_RM<\infty$.
\end{cor}
\begin{proof}
If $\dim R = 0$, then $M$ has finite length and we are done by the rigidity of $\ps R$.
Hence we may assume that $\dim R > 0$.

It follows from Lemma~\ref{torsyz} and Corollary~\ref{torcm} that
$\lim\limits_{n\to \infty} \frac{\length \left (\Tor_i^R(M,\ps R)\right)}{p^{n\dim R}} = \hk(\syz^{i} M).$
Since $\depth (\syz^i M) > 0$, Corollary~\ref{posci} implies that
$\pd_R M < \infty$ if and only if $\hk(\syz^i M) = 0$.
\end{proof}
 
\begin{thm}\label{homcplx}
Let $(R, \m)$ be a local complete intersection of dimension $d > 0$ 
and $F_{\bullet}$ be a complex of finite free modules. 
Let us denote $C_i = \cok(F_i \to F_{i-1})$ and 
assume that , for some $i$,
$\IPD (C_i) \subseteq \{\m\}$ and $H_i(F_{\bullet})$ has finite length. 
Then for $G(X) = \CH_i(F_{\bullet}\otimes X)$ we have
\[\lim\limits_{n\to \infty} \frac{\length(G(\ps R))}{p^{nd}} =0\]
if and only if  $H_i(F_{\bullet})=0$ and $\pd_R C_i <\infty$.
\end{thm}
\begin{proof}
By a theorem of Auslander and its proof (see \cite[Proposition 3.6]{Ha}) there exists an exact sequence of functors
\begin{equation}\label{aus}
\Tor_2^R (C_i, X) \to G(R)\otimes X \to G(X) \to \Tor_1^R (C_i, X) \to 0.
\end{equation}
If $\pd_R C_i < \infty$, then both $\Tor$-modules vanish by Theorem~\ref{tPeskine-Szpiro} 
and one direction follows.

For the converse, if $\lim\limits_{n\to \infty} \frac{\length(G(\ps R))}{p^{nd}} = 0$ then  
$\lim\limits_{n\to \infty} \frac{\length\left (\Tor_1^R(C_i, \ps R) \right)}{p^{nd}} = 0.$
Therefore, $\pd_R C_i < \infty$ and both Tor-modules vanish by Corollary~\ref{hyprigid},
so we have an isomorphism $G(R)\otimes \ps R \cong G(\ps R)$.
But 
\[\lim_{n\to \infty} \frac{\length(G(R) \otimes \ps R)}{p^{nd}} = 0\]
if and only if $G(R) = 0$, {\it i.e.}, when $H_i(F_\bullet) = 0$.
\end{proof}

\begin{cor}\label{Extci}
Let $(R, \m)$ be a local complete intersection of dimension $d > 0$ 
and $M$ be a finitely generated $R$-module such that $\IPD(M) \inc \{\m\}$. Let $\depth M \geq k$.
Then 
\[
\lim_{n \to \infty} \frac{\length \left(\lc^k_{\m}(F^n(M)) \right)}{p^{nd}} = 0  \text{ if and only if }
\lim_{n\to \infty} \frac{\length \left(\Ext^{d - k}_R(M,\ps R) \right)}{p^{nd}} =0\] 
if and only if 
$\pd_R M < d - k.$
\end{cor}

\begin{proof}
The first two assertions are equivalent by Theorem~\ref{extlimit}.
When $\pd_R M < \dim R - k$, the limit is zero by Lemma~\ref{finpdbig}.

For the remaining direction, set $N = \syz^{d - k -1} M$. Then
$N$ is a maximal Cohen-Macaulay module such that  
$\Ext^{d- k}_R(M,\ps R) = \Ext^1_R(N, \ps R)$.
By Theorem~\ref{homcplx}, the first cosyzygy of $N^*$ has finite projective dimension, thus $N$ is free. 
Therefore, $\pd M < \infty$.
\end{proof}

\begin{rmk}
The assumption on $\depth M$ is necessary.
Let  $(R, \m)$ be a regular local ring of dimension $d > 1$ and 
$M =  R \oplus k$. Note that $\pd M = d$, but 
$\lc^1_{\m}(F^n(M)) = \lc^1_{\m} (R \oplus R/\m^{[p^n]}) = 0$ for all $n$.
\end{rmk}

Next we discuss an application on local cohomology of symbolic powers of reflexive ideals. We start with a simple, and perhaps well-known result. 
\begin{prop}
Let $(R, \m)$ be a local ring which satisfies Serre's condition $(S_2)$ and $d=\dim R\geq 3$. Let $M$ be a finitely generated  $R$-module which is locally free on $\Spec R-\{\m\}$.  Then $\lc^2_{\m}(M) \cong \lc^2_{\m}(M^{**})$ and they have finite length. 
\end{prop}

\begin{proof}
There is a natural map $M \to M^{**}$ which has finite length kernel and cokernel. Thus one can easily use the exact sequence of local cohomology to show that $\lc^2_{\m}(M) \cong \lc^2_{\m}(M^{**})$. 

To show that these modules have finite length, 
we can complete and assume that $R$ has a dualizing complex $D_R$. By the local duality, we have to show that $\CH^{d-2}(\Hom(M,D_R))$ has finite length. 
Moreover, localizing at any prime $\p \neq \m$, we only need to concern with the case $\height \p\geq d-2$. 
But then $M_\p \cong R_\p^n$, so
$\CH^{d-2}(\Hom(M,D_R))_{\p}$ is dual to 
$\lc^{\height \p-d+2}_{\p R_{\p}}(R_{\p}^n)$.
However, since $ \height \p-d+2 \leq 1$ and $R$ is $(S_2)$,
$\lc^{\height \p-d+2}_{\p R_{\p}}(R_{\p}^n) = 0$.
\end{proof}

For an ideal $I$, recall that $I^{(n)}$ denotes the $n$th symbolic power of $I$. The second part of the following Theorem is an effective version of the main result of \cite{DLM}. 
Namely, \cite[Theorem~2.15]{DLM} assumes that $\lc^2_{\m}(I^{(p^n)}) = 0$ 
for some $n$, while we assume that the limit is zero. 

\begin{thm}\label{refl}
Let $(R, \m)$ be a local ring satisfying Serre's condition $(S_2)$ and $d=\dim R\geq 3$. Let $I$ be a reflexive ideal that is locally free on $\Spec R-\{\m\}$. Then 
\begin{enumerate}
\item There exist elements $a,b\in R$ such that 
$| \length(\lc^2_{\m}(I^{(p^n)})) - \length({\lc^0_{\m}}(R/(a,b)^{[p^n]})) |$ 
is bounded by a constant. In particular,
$\lim\limits_{n \to \infty} \frac{\length(\lc^2_{\m}(I^{(p^n)}))}{p^{nd}} = 
\lim\limits_{n \to \infty}  \frac{\length(\lc^2_{\m}(I^{\otimes p^n}))}{p^{nd}} = \hk(R/(a,b)).$ 
\item When $R$ is a complete intersection, the limit in part $(1)$ is $0$ if and only if $I$ is principal. In particular, the Picard group of  $\Spec R-\{\m\}$ has no non-trivial torsion elements.
\end{enumerate}
\end{thm}

\begin{proof}
First we note that  $\length(\lc^2_{\m}(I^{(p^n)})) = \length (\lc^2_{\m}(I^{\otimes p^n}))$ follows from the previous Proposition and the fact that $I^{(n)} \cong (I^{\otimes n})^{**}$.

The proof of \cite[Theorem 2.9]{DLM} shows that there are elements $a, b$ such that
$I=(a):(b)$ and, furthermore, $I^{(n)}=(a^n):b^n$.
Taking local cohomology of 
$
0 \to R/(a^n:b^n) \to R/(a^n) \to R/(a^n,b^n) \to 0
$
we get an exact sequence 
\[0 \to \lc_{\m}^0(R/(a,b)^{[p^n]}) \to \lc^1_{\m}(R/I^{(p^n)}) \to \lc^1_{\m}(R/(a^{p^n})),\]
which shows that
\[
0 \leq \length (\lc^1_{\m}(R/I^{(p^n)})) -  \length(\lc_{\m}^0(R/(a,b)^{[p^n]})) \leq 
\length(\lc^1_{\m}(R/(a^{p^n}))) \leq \length (\lc^2_{\m}(R)) < \infty,
\]
where the second inequality follows by taking local cohomology of 
$\ses{R}{R}{R/(a^{p^n})}$. 
Similarly, by taking local cohomology of $\ses{I^{(p^n)}}{R}{R/I^{(p^n)}}$ we may get that 
\[
0 \leq \length(\lc^2_{\m}(I^{(p^n)})) - \length(\lc^1_{\m}(R/I^{(p^n)})) \leq \length(\lc^2_{\m}(R))
\] 
and the first statement follows. Note that $\hk(R/(a,b))$ exists by 
Corollary \ref{34}.

For the second part consider the exact sequence 
\[0 \to \Tor_1^R (R/I, \ps R) \to F^n(I) \to  R \to R/I^{[p^n]} \to 0\]
arising from the tensor product with $\ps R$. 
Because  $\Tor_1^R (R/I, \ps R)$ has finite length, the sequence shows that
$\lc^2_{\m}(F^n(I)) \cong \lc^2_{\m}(I^{[p^n]}) = \lc^2_{\m}(I^{(p^n)})$ 
and we may apply Corollary~\ref{Extci}. 
If $I$ represents a torsion elements in the Picard group of  $\Spec R-\{\m\}$, then $\lc^2_{\m}(I^{(p^n)})$ must be periodic, hence 
$
\lim\limits_{n \to \infty} \frac{\length \left(\lc^2_{\m}(I^{(p^n)}) \right)}{p^{3n}}=0.
$

\end{proof}

\begin{eg}\label{need gor}
We can use Theorem~\ref{refl} to demonstrate that Corollary~\ref{posci}
may not hold if $R$ is merely Gorenstein. 
Namely, it is easy to find examples of Gorenstein isolated singularity with torsion Picard group, e.g.,  
a suitable Veronese subring (\cite[Example 3.2]{DLM}).
\end{eg}

\section{More precise behavior of  $\fhk M(n)$}\label{pre}

Numerical experiments suggest that when $R$ has an isolated singularity, the behavior of $\fhk M(n)$ follows the case of classical Hilbert--Kunz functions. We discuss this phenomenon and establish some  special cases.

Let $M$ be a finite module over a local hypersurface $(R, \m)$ of dimension $d$. 
If $\IPD(M) \inc \{\m\}$, then we can define Hochster's theta function (\cite{Hochster}) 
\[\theta^R(M, X) = \length(\Tor_{2d}^R(M, X)) -\length (\Tor_{2d+1}^R(M, X))\]
for any module $X$.
The following lemma is implicitly contained in~\cite{Dao}. 
We place a proof here for convenience, but refer to the paper for more information.

\begin{lemma}\label{theta}
Let $(R, \m)$ be a local hypersurface and $M$ be a finite $R$-module with  $\IPD(M) \inc \{\m\}$.
Then $\theta (M, \ps R) = 0$.
\end{lemma}
\begin{proof}
We consider $\theta (M, -)$ as a function on the Grothendieck group of $R$.	
By \cite[Remark 2.8]{Kurano}) $[\ps R] = p^{n\dim R}[R]$ in the Grothendieck group, so
$\theta  (M, \ps R) = p^{n\dim R} \theta (M, R) = 0$.
\end{proof}

\begin{thm}\label{pre1}
Let $(R, \m)$ be a local hypersurface with perfect residue field  
and $M$ be a finite $R$-module with $\depth M>0$ and $\IPD(M) \inc \{\m\}$. 
Then there are modules $M_1,M_2$ of finite length such that
\[\fhk M(n) = \frac{\fhk {M_1}(n)-\fhk {M_2}(n)}{2}.\]
\end{thm}

\begin{proof}
The Auslander--Buchweitz approximation (\cite[1.8]{ABu}) of $M$ is an exact sequence
$\ses MQL,$
where $\pd Q < \infty$ and $L$ is maximal Cohen-Macaulay.
Tensoring with $\ps R$ we obtain the exact sequence
\[0 \to \Tor_1^R(L, \ps R)\to F^n(M) \xrightarrow{g} F^n(Q)\to F^n(L) \to 0.\]
Since $\Tor_1^R(L, \ps R)$ has finite length,  we have the exact sequence of local cohomology
\[0 \to \Tor_1^R (L, \ps R) \to \lc^0_\m (F^n(M)) \to \lc^0_\m (F^n(Q)) = 0,\] 
where we used that $\depth F^n(Q) = \depth Q > 0$.
Hence it is enough to prove the statement replacing $\fhk M$ by $\length(\Tor_1^R(L,\ps R))$.  

Since $\IPD(M) \subseteq \{\m\}$, $L$ is locally free on the punctured spectrum.
Thus there is an $\m$-primary ideal $I$ ({\it e.g.}, by adapting the proof of \cite[Lemma~4.3]{DV}) 
that kills all functors $\Tor_i^R(L,-)$ for $i>0$.
Choose an $L$-regular element $x \in I$, then the sequence 
$0 \to L \xrightarrow{x} L \to L/xL\to 0$ 
gives that:
\[0 \to  \Tor_{i+1}^R(L,\ps R) \to \Tor_{i+1}^R(L/xL,\ps R) \to \Tor_{i}^R(L,\ps R) \xrightarrow{x} 0 \]
for $i\geq 1$. 

By the previous lemma and the fact that the minimal resolution of $L$ is $2$-periodic (\cite[Theorem~6.1]{Eisenbud}), 
we get that $\length(\Tor_{i+1}^R(L/xL,\ps R)) = 2\length(\Tor_{i}^R(L,\ps R))$.
Now, $L/xL$ still has finite projective dimension on the punctured spectrum and its minimal resolution is still $2$-periodic for $i>1$, so
we can choose another element $y\in I$ regular on $L/xL$ and apply the same argument.
This way, we continue choosing $L$-regular sequence $\ul{x}$ in $I$ to get 
that   $\length(\Tor_{i+d}^R(L/\ul{x}L,\ps R)) = 2^d\length(\Tor_{i}^R(L,\ps R))$ and the result follows from the next lemma.
\end{proof}

\begin{lemma}
Let $R$ be a Cohen-Macaulay local ring and $M$ be a module of finite length. Then for each $i > 0$ there are finite length modules $M_1,M_2$ such that
\[\length(\Tor_i^R(M,\ps R)) = \fhk {M_1}(n)-\fhk {M_2}(n).\]

\end{lemma}
\begin{proof}
Let $\underline x$ be a regular sequence in the annihilator of $M$. Then there is an exact sequence 
\[\ses{N}{(R/\underline xR)^m}{M}.\]
If $i = 1$,  tensoring with  $\ps R$, we obtain the exact sequence
\[0 \to \Tor_1^R(M, \ps R) \to F^n(N) \to F^n((R/\ul{x})^m) \to F^n(M) \to 0,\]
so we can take $M_1=M\oplus N$ and $M_2=(R/\underline xR)^m$. 
For $i > 1$, we get that $\Tor_{i}^R(M,\ps R) \cong \Tor_{i-1}^R(N,\ps R)$, and the result follows by induction. 
\end{proof}

%

\specialsection*{ACKNOWLEDGEMENTS}
We thank Craig Huneke and University of Virginia for creating the opportunity for us to work together! 
We thank Holger Brenner, Srikanth Iyengar, Ryo Takahashi and Kei-ichi Watanabe for many helpful conversations. 
We thank the anonymous referee for valuable comments.

\bibliographystyle{alpha}
\bibliography{gHK}

\begin{thebibliography}{DLM10}

\bibitem[AB69]{ABR}
Maurice Auslander and Mark Bridger.
\newblock {\em Stable module theory}.
\newblock Memoirs of the American Mathematical Society, No. 94. American
  Mathematical Society, Providence, R.I., 1969.

\bibitem[AB89]{ABu}
Maurice Auslander and Ragnar-Olaf Buchweitz.
\newblock The homological theory of maximal {C}ohen-{M}acaulay approximations.
\newblock {\em M\'em. Soc. Math. France (N.S.)}, (38):5--37, 1989.
\newblock Colloque en l'honneur de Pierre Samuel (Orsay, 1987).

\bibitem[Abe08]{Ab}
Ian~M. Aberbach.
\newblock The existence of the {F}-signature for rings with large {$\Bbb
  Q$}-{G}orenstein locus.
\newblock {\em J. Algebra}, 319(7):2994--3005, 2008.

\bibitem[AF94]{AF}
Luchezar~L. Avramov and Hans-Bj{\o}rn Foxby.
\newblock Grothendieck's localization problem.
\newblock In {\em Commutative algebra: syzygies, multiplicities, and birational
  algebra ({S}outh {H}adley, {MA}, 1992)}, volume 159 of {\em Contemp. Math.},
  pages 1--13. Amer. Math. Soc., Providence, RI, 1994.

\bibitem[AL08]{AL}
Ian~M. Aberbach and Jinjia Li.
\newblock Asymptotic vanishing conditions which force regularity in local rings
  of prime characteristic.
\newblock {\em Math. Res. Lett.}, 15(4):815--820, 2008.

\bibitem[AM01]{AM}
Luchezar~L. Avramov and Claudia Miller.
\newblock Frobenius powers of complete intersections.
\newblock {\em Math. Res. Lett.}, 8(1-2):225--232, 2001.

\bibitem[BC18]{BreCa}
Holger Brenner and Alessio Caminata.
\newblock Generalized {H}ilbert-{K}unz function in graded dimension 2.
\newblock {\em Nagoya Math. J.}, 230:1--17, 2018.

\bibitem[Bre]{Bre}
Holger Brenner.
\newblock Irrational {H}ilbert-{K}unz multiplicities.
\newblock Preprint, available at http://arxiv.org/abs/1305.5873.

\bibitem[Bur72]{Burch}
Lindsay Burch.
\newblock Codimension and analytic spread.
\newblock {\em Proc. Cambridge Philos. Soc.}, 72:369--373, 1972.

\bibitem[Che44]{Chev}
Claude Chevalley.
\newblock Some properties of ideals in rings of power series.
\newblock {\em Trans. Amer. Math. Soc.}, 55:68--84, 1944.

\bibitem[Cut13]{Cut}
Steven~Dale Cutkosky.
\newblock Multiplicities associated to graded families of ideals.
\newblock {\em Algebra Number Theory}, 7(9):2059--2083, 2013.

\bibitem[Dao13]{Dao}
Hailong Dao.
\newblock Decent intersection and {T}or-rigidity for modules over local
  hypersurfaces.
\newblock {\em Trans. Amer. Math. Soc.}, 365(6):2803--2821, 2013.

\bibitem[DLM10]{DLM}
Hailong Dao, Jinjia Li, and Claudia Miller.
\newblock On the (non)rigidity of the {F}robenius endomorphism over
  {G}orenstein rings.
\newblock {\em Algebra Number Theory}, 4(8):1039--1053, 2010.

\bibitem[DS]{DS}
Hailong Dao and Tony Se.
\newblock Finite {F}-type and {F}-abundant modules.
\newblock Preprint, available at https://arxiv.org/abs/1603.00334.

\bibitem[Dut83]{D1}
Sankar~P. Dutta.
\newblock Frobenius and multiplicities.
\newblock {\em J. Algebra}, 85(2):424--448, 1983.

\bibitem[Dut89]{D2}
Sankar~P. Dutta.
\newblock Ext and {F}robenius.
\newblock {\em J. Algebra}, 127(1):163--177, 1989.

\bibitem[DV09]{DV}
Hailong Dao and Oana Veliche.
\newblock Comparing complexities of pairs of modules.
\newblock {\em J. Algebra}, 322(9):3047--3062, 2009.

\bibitem[DW16]{DW}
Hailong Dao and Kei-ichi Watanabe.
\newblock Some computations of the generalized {H}ilbert-{K}unz function and
  multiplicity.
\newblock {\em Proc. Amer. Math. Soc.}, 144(8):3199--3206, 2016.

\bibitem[Eis80]{Eisenbud}
David Eisenbud.
\newblock Homological algebra on a complete intersection, with an application
  to group representations.
\newblock {\em Trans. Amer. Math. Soc.}, 260(1):35--64, 1980.

\bibitem[EY17]{EY}
Neil Epstein and Yongwei Yao.
\newblock Some extensions of {H}ilbert-{K}unz multiplicity.
\newblock {\em Collect. Math.}, 68(1):69--85, 2017.

\bibitem[Har66]{Hartshorne}
Robin Hartshorne.
\newblock {\em Residues and duality}.
\newblock Lecture notes of a seminar on the work of A. Grothendieck, given at
  Harvard 1963/64. With an appendix by P. Deligne. Lecture Notes in
  Mathematics, No. 20. Springer-Verlag, Berlin-New York, 1966.

\bibitem[Har98]{Ha}
Robin Hartshorne.
\newblock Coherent functors.
\newblock {\em Adv. Math.}, 140(1):44--94, 1998.

\bibitem[HJ18]{HeJe}
Daniel~J. Hern\'{a}ndez and Jack Jeffries.
\newblock Local {O}kounkov bodies and limits in prime characteristic.
\newblock {\em Math. Ann.}, 372(1-2):139--178, 2018.

\bibitem[Hoc81]{Hochster}
Melvin Hochster.
\newblock The dimension of an intersection in an ambient hypersurface.
\newblock In {\em Algebraic geometry ({C}hicago, {I}ll., 1980)}, volume 862 of
  {\em Lecture Notes in Math.}, pages 93--106. Springer, Berlin-New York, 1981.

\bibitem[Hun00]{Hu2}
Craig Huneke.
\newblock The saturation of {F}robenius powers of ideals.
\newblock {\em Comm. Algebra}, 28(12):5563--5572, 2000.
\newblock Special issue in honor of Robin Hartshorne.

\bibitem[Hun13]{Hu}
Craig Huneke.
\newblock Hilbert-{K}unz multiplicity and the {F}-signature.
\newblock In {\em Commutative algebra}, pages 485--525. Springer, New York,
  2013.

\bibitem[Kat96]{Katzman}
Mordechai Katzman.
\newblock Finiteness of {$\bigcup_e{\rm Ass}\,F^e(M)$} and its connections to
  tight closure.
\newblock {\em Illinois J. Math.}, 40(2):330--337, 1996.

\bibitem[Kun76]{Kunz}
Ernst Kunz.
\newblock On {N}oetherian rings of characteristic {$p$}.
\newblock {\em Amer. J. Math.}, 98(4):999--1013, 1976.

\bibitem[Kur96]{Kurano}
Kazuhiko Kurano.
\newblock A remark on the {R}iemann-{R}och formula on affine schemes associated
  with {N}oetherian local rings.
\newblock {\em Tohoku Math. J. (2)}, 48(1):121--138, 1996.

\bibitem[KV10]{JK}
Daniel Katz and Javid Validashti.
\newblock Multiplicities and {R}ees valuations.
\newblock {\em Collect. Math.}, 61(1):1--24, 2010.

\bibitem[Li08]{Li}
Jinjia Li.
\newblock Characterizations of regular local rings in positive characteristics.
\newblock {\em Proc. Amer. Math. Soc.}, 136(5):1553--1558, 2008.

\bibitem[Mil00]{M1}
Claudia Miller.
\newblock A {F}robenius characterization of finite projective dimension over
  complete intersections.
\newblock {\em Math. Z.}, 233(1):127--136, 2000.

\bibitem[Mil03]{M2}
Claudia Miller.
\newblock The {F}robenius endomorphism and homological dimensions.
\newblock In {\em Commutative algebra ({G}renoble/{L}yon, 2001)}, volume 331 of
  {\em Contemp. Math.}, pages 207--234. Amer. Math. Soc., Providence, RI, 2003.

\bibitem[Mon83]{Monsky}
Paul Monsky.
\newblock The {H}ilbert-{K}unz function.
\newblock {\em Math. Ann.}, 263(1):43--49, 1983.

\bibitem[PS69]{PS1}
Christian Peskine and Lucien Szpiro.
\newblock Sur la topologie des sous-sch\'emas ferm\'es d'un sh\'ema localement
  noeth\'erien, d\'efinis comme support d'un faisceau coh\'erent localement de
  dimension projective finie.
\newblock {\em C. R. Acad. Sci. Paris S\'er. A-B}, 269:A49--A51, 1969.

\bibitem[PS73]{PS2}
C.~Peskine and L.~Szpiro.
\newblock Dimension projective finie et cohomologie locale. {A}pplications \`a
  la d\'emonstration de conjectures de {M}. {A}uslander, {H}. {B}ass et {A}.
  {G}rothendieck.
\newblock {\em Inst. Hautes \'Etudes Sci. Publ. Math.}, (42):47--119, 1973.

\bibitem[Sch98]{Schenzel}
Peter Schenzel.
\newblock On the use of local cohomology in algebra and geometry.
\newblock In {\em Six lectures on commutative algebra ({B}ellaterra, 1996)},
  volume 166 of {\em Progr. Math.}, pages 241--292. Birkh\"auser, Basel, 1998.

\bibitem[Sei89]{Seibert}
Gerhard Seibert.
\newblock Complexes with homology of finite length and {F}robenius functors.
\newblock {\em J. Algebra}, 125(2):278--287, 1989.

\bibitem[Vra16]{Vra}
Adela Vraciu.
\newblock An observation on generalized {H}ilbert-{K}unz functions.
\newblock {\em Proc. Amer. Math. Soc.}, 144(8):3221--3229, 2016.

\bibitem[Wei94]{Weibel}
Charles~A. Weibel.
\newblock {\em An introduction to homological algebra}, volume~38 of {\em
  Cambridge Studies in Advanced Mathematics}.
\newblock Cambridge University Press, Cambridge, 1994.

\bibitem[Yos90]{Yo}
Yuji Yoshino.
\newblock {\em Cohen-{M}acaulay modules over {C}ohen-{M}acaulay rings}, volume
  146 of {\em London Mathematical Society Lecture Note Series}.
\newblock Cambridge University Press, Cambridge, 1990.

\end{thebibliography}

\appendix \section{Numerical evidences and examples}

In this section we collect and comment on some numerical evidences involving $\hk(M)$. Most of our examples here were performed using the software MACAULAY 2. This numerical data helped us formulate the results in Section \ref{pre} and in addition  has inspired several works since our preprint became available (\cite{BreCa, DS, DW, Vra}). 

\begin{eg}
Let $R = k[[x,y,u,v]]/(xy-uv)$ where $k$ is algebraically closed of characteristic $p>2$. 
For a module $M$ of positive depth, We can compute $\fhk M(n)$ as follows:

Let $M_1 = (x, y)$ and $M_2 = (x, v)$.
It is known that (see for example \cite{Yo})  $M_1$ and $M_2$ are the only nonfree indecomposable maximal Cohen-Macaulay modules.
Moreover, there exists a decomposition
\[\ps R = R^{a_n} \oplus M_1^a \oplus M_2^b,\]
where $a = b = \frac{p^{3n} - a_n}{2}$.

One can check that $\Tor_1^R (M_1, M_2) = 0$ and $\Tor_1^R(M_1, M_1) = \Tor_1^R(M_2, M_2) = 1$.
Let $N = R^m \oplus M_1^c \oplus M_2^d$ be a decomposition of a maximal Cohen-Macaulay module $N$
obtained from a maximal Cohen-Macaulay approximation of $M$.
We know that 
\[\length (\lc^0_{\m}(F^n(M)) = \length (\Tor_1^R(N, \ps R)) = ca + db = (c + d)\frac{p^{3n} - a_n}{2}.\]

\end{eg}

\begin{eg}
There is some numerical evidence that over a normal graded domain of dimension two,  
$\fhk M(n) = cp^{2n} + \gamma(n)$ where $\gamma$ is a bounded (even periodic) function. 
For example:
\begin{enumerate}
\item Consider $R=k[[x,y,z]]/(x^3+y^3+z^3)$ and $M=R/(x,y+z)$. 
Macaulay2 calculations suggest that $\fhk M (n) = \frac{4(p^{2n}-1)}{3}$ for $k = \mathbb F_2, F_5, F_7, F_{11}$. Compare with $\fhk k (n) =  \frac{9p^{2n}-5}{4}$!

\item For $R=\mathbb F_3[[x,y,z]]/(x^4+y^4-z^4)$, $M=R/(x,y^2-z^2)$, and $N= R/(x,y-z)$
Macaulay2 suggests that $\fhk M (n) = 3 (3^{2n}-1)$ and $\fhk N (n) = \frac{9(3^{2n}-1)}{4}$.

\item For $R=\mathbb F_2[[x,y,z]]/(x^5+y^5-z^5)$(or $\mathbb F_3$) and  $M=R/(x,y-z)$, 
computations show that $\fhk M (n) =\frac{16p^{2n}-\gamma(n)}{5}$ with $\gamma(n) = 24$ for odd $n$ and $= 16$ for even $n$. 
\end{enumerate}
\end{eg}

\begin{eg}\label{nonis}
Here are some examples showing that the behavior of $\fhk M$ is not similar 
to the finite length case without the assumption on $\IPD(M)$. 
The point is that the second coefficient is not $0$ as is the case for dimension $2$ and classical  Hilbert--Kunz  functions. 
\begin{enumerate}
\item When $R=k[[x,y,t]]/(x^4+tx^2y^2+y^4)$, and $M=R/(x,y)$, $\fhk M= p^n-2$. 
\item When $R=k[[x,y,t]]/(x^4+txy^2+y^4)$, 
and $M=R/(x,y)$, $\fhk M= 2^{2n}/4+2^n-2$ for $k=\mathbb F_2$ and $n>1$. 
For $k= \mathbb F_3, \mathbb F_5$, it is $(3^{2n}-13)/4+3^n$.

\item Let $R=k[[x,y,t]]/(x^3+txy+y^3)$, and $M=R/(x,y)$. 
Then for $k=\mathbb F_2$, $\fhk M= \frac{2^{2n}+2 \cdot 2^n-\gamma(n)}{3}$. 
For $k=\mathbb F_3$, it is $ \frac{3^{2n}+2 \cdot 3^n-3}{3}$. 
For $k=\mathbb F_5$, it is $\frac{5^{2n}+2 c\dot 5^n-\gamma(n)}{3}$.
For $k=\mathbb F_7$, it is $ \frac{7^{2n}+2 \cdot 7^n-3}{3}$.  
For $k=\mathbb F_{11}$, it is $\frac{(11)^{2n}+2 (11)^n-\gamma(n)}{3}$.
Where $\gamma(n) = 3$ for odd $n$ and $= 5$ for even.
The formula seems to depend on whether $q=1 \mod 3$.
\item When $R=\mathbb F_2[[x,y,t]]/(x^3+txy+y^3)$, and $M=R/(x^3,y^3)$, we get $3\cdot 2^{2n}+2\cdot 2^n-1$. 
For $M=R/(x^2,y^2,xy)$ we get $2^{2n}+2^n-2$.
\end{enumerate}
\end{eg}

\end{document}